\documentclass{mymathart}
\usepackage[theorems]{mymathmacros}

\renewcommand{\A}{\mathcal{A}}
\newcommand{\Af}{\A_f}
\newcommand{\Afl}{\Af^{\ell}}
\newcommand{\Afn}{\Af^{n}}

\newcommand{\U}{\mathcal{U}}
\newcommand{\Ua}{\U^{a}}

\renewcommand{\R}{\mathcal{R}}
\newcommand{\Rf}{\R_f}

\newcommand{\z}{\hat{z}}
\newcommand{\zetah}{\hat{\zeta}}

\newcommand{\fh}{\hat{f}}
\newcommand{\f}{\fh}

\newcommand{\Ah}{\hat{A}}

\renewcommand{\Z}{\mathcal{Z}}

\usepackage{stmaryrd}

\newcommand{\bdyit}[2]
             {{\rule{0pt}{0pt}_{\mbox{$\scriptstyle #2$}}^{\mbox{%
                   $\scriptstyle #1$}} }}

\newcommand{\Crit}{\operatorname{Crit}}

\renewcommand{\P}{\mathcal{P}}

\newcommand{\Nf}{\mathcal{N}_f}

\newcommand{\lam}{\operatorname{lam}}

\newcommand{\psih}{\hat{\psi}}
\newcommand{\gammah}{\hat{\gamma}}

\renewcommand{\phi}{\varphi}

\newcommand{\univ}{\operatorname{Univ}}

\usepackage{hyperref}

\title[Hyperbolic Leaves and Wild 
         Laminations]{%
   A note on \\ 
     hyperbolic leaves and wild laminations \\ 
    of rational functions}

\author{Jeremy Kahn}
\address{Institute for Mathematical Sciences, SUNY Stony Brook, 
               NY 11794-3660, USA}
\email{kahn@math.sunysb.edu}
\author{Mikhail Lyubich}
\address{Institute for Mathematical Sciences, SUNY Stony Brook, 
               NY 11794-3660, USA}
\email{mlyubich@math.sunysb.edu}
\author{Lasse Rempe}
\address{Dept.\ of Math.\ Sciences, University of Liverpool, Liverpool L69 7ZL, UK}
\email{l.rempe@liverpool.ac.uk}
\thanks{ The second author has been partially supported by the
  NSF and NSERC.
The third author has been partially supported
 by a postdoctoral fellowship of the 
 German Academic Exchange Service (DAAD), and later by EPSRC
 fellowship EP/E052851/1. }

\begin{document}

\def\IMSmarkvadjust{0 pt}
\def\IMSmarkhadjust{0 pt}
\def\IMSmarkhpadding{0 pt}
\def\IMSpubltext{Published in modified form:}
\def\SBIMSMark#1#2#3{
 \font\SBF=cmss10 at 10 true pt
 \font\SBI=cmssi10 at 10 true pt
 \setbox0=\hbox{\SBF \hbox to \IMSmarkhpadding{\relax}
                Stony Brook IMS Preprint \##1}
 \setbox2=\hbox to \wd0{\hfil \SBI #2}
 \setbox4=\hbox to \wd0{\hfil \SBI #3}
 \setbox6=\hbox to \wd0{\hss
             \vbox{\hsize=\wd0 \parskip=0pt \baselineskip=10 true pt
                   \copy0 \break%
                   \copy2 \break% 
                   \copy4 \break}}
 \dimen0=\ht6   \advance\dimen0 by \vsize \advance\dimen0 by 8 true pt
                \advance\dimen0 by -\pagetotal
	        \advance\dimen0 by \IMSmarkvadjust
 \dimen2=\hsize \advance\dimen2 by .25 true in
	        \advance\dimen2 by \IMSmarkhadjust

%
%   Check for publication info
%
%  \newread\jref
  \openin2=publishd.tex
  \ifeof2\setbox0=\hbox to 0pt{}
  \else 
     \setbox0=\hbox to 3.1 true in{
                \vbox to \ht6{\hsize=3 true in \parskip=0pt  \noindent  
                {\SBI \IMSpubltext}\hfil\break
                \input publishd.tex 
                \vfill}}
  \fi
  \closein2
  \ht0=0pt \dp0=0pt
 \ht6=0pt \dp6=0pt
 \setbox8=\vbox to \dimen0{\vfill \hbox to \dimen2{\copy0 \hss \copy6}}
 \ht8=0pt \dp8=0pt \wd8=0pt
 \copy8
 \message{*** Stony Brook IMS Preprint #1, #2. #3 ***}
}

\SBIMSMark{2008/5}{October 2008}{}

\begin{abstract}
 We study the affine orbifold laminations that were 
  constructed in \cite{mishayair}. An important question left open in
  \cite{mishayair} is whether these laminations are always locally
  compact. We show that this is not the case.

 The counterexample we construct has the property that the
  \emph{regular leaf space} contains (many) hyperbolic leaves
  that intersect the Julia set; whether this can happen is itself
  a question raised in \cite{mishayair}.
\end{abstract}

 \maketitle

 \section{Introduction}
 Providing a new line in the ``Sullivan dictionary'' between
  rational maps and Kleinian groups,
  the article \cite{mishayair} associated an ``Affine Orbifold Lamination''
  $\Af$
  to any rational map $f:\Ch\to\Ch$. (See Section \ref{sec:preliminaries} for
  an overview of the definitions.) 

 This lamination is particularly useful when it is locally 
  compact; e.g.\ this condition allows the construction of  
  transverse 
  conformal measures and invariant measures
  on the lamination \cite{mishakaim}. Local compactness
  is satisfied in certain important cases, including
  geometrically finite rational maps and Feigenbaum-like quadratic
  polynomials.\footnote{%
  We should note that there is a different construction of laminations
  for rational maps, due to
  Meiyu Su \cite{sulaminations}. These laminations are never locally
  compact.} %
 The question whether the lamination $\Af$ is always locally
 compact was raised in \cite[\S 10, Question 9]{mishayair}.

 \begin{thm}[Failure of local compactness]
   \label{thm:main}
  There exists a quadratic polynomial $f$ whose affine
   orbifold lamination $\Af$ is not locally compact. 
 \end{thm}

  Our proof of Theorem \ref{thm:main} is related to another question
   from \cite{mishayair}. As we review in Section \ref{sec:preliminaries},
   the \emph{regular leaf space} $\Rf$ consists of those backward orbits under
   $f$ along which some disk can be pulled back with a bounded amount of
   branching. The path-connected components (\emph{leaves}) 
   of $\Rf$ have a natural
   Riemann surface structure; in many cases all such leaves are 
   parabolic planes. It was asked in \cite[\S 10, Question 2]{mishayair}
   whether rational maps can have leaves that are hyperbolic but 
   do not arise from Siegel disks or Herman rings. Hubbard 
   (personal communication) was the first to 
   suggest an example with this property 
   (see the remark at the end of Section \ref{sec:hypleaves}). The
   hyperbolic leaves of this example lie over a single Fatou component,
   so the question remained 
   whether hyperbolic leaves can intersect 
   the Julia set. We give a positive answer. 

 \begin{thm}[Hyperbolic leaves intersecting the Julia set]
  \label{thm:quadratichyperbolic}
  There exists a quadratic polynomial
  whose regular leaf space contains a hyperbolic leaf that intersects
  the Julia set.
 \end{thm}
 \begin{remark}
   Rivera-Letelier (personal communication) has announced a
    proof of the stronger result that this polynomial
    can be chosen to satisfy the    topological 
    Collet-Eckmann condition.
 \end{remark}

  We then use the following result to deduce Theorem \ref{thm:main}
   from Theorem \ref{thm:quadratichyperbolic}.

 \begin{thm}[Hyperbolic leaves and local compactness] 
   \label{thm:localcompactness}
  Let $f$ be a rational function whose regular
  leaf space contains a hyperbolic leaf that intersects the
  Julia set. Then the affine orbifold lamination
   $\Af$ is not locally compact.
 \end{thm}

 \medskip

 \noindent\textsc{Acknowledgements. }
  We thank Carlos Cabrera and
   Juan Rivera-Letelier for useful discussions.
   The third author would like to thank the Institute
   for Mathematical Sciences at Stony Brook and the Simons endowment 
   for its continued support and hospitality.

 \section{Preliminaries} \label{sec:preliminaries}

 In this section, we introduce basic notations and
  give an account of the construction
  of $\Af$ that is sufficient to provide a self-contained proof of
  our results. This account will necessarily be kept concise; 
  for more details, we refer the reader to \cite{mishayair}. 

 \subsection*{Basic definitions}
   The complex plane and Riemann sphere are denoted $\C$ and $\Ch$, as usual. 
    A (spherical) disk of radius $\eps$ around $z\in\Ch$ is denoted
    $D_{\eps}(z)$. If $V$ and $U$ are open sets such that
    $\cl{V}$ is a compact subset of $U$, then we say that
    $V$ is \emph{compactly contained} in $U$ and write
    $V\Subset U$. 

  Throughout this article, $f:\Ch\to\Ch$ will be a rational
  endomorphism of the Riemann sphere. As usual, $F(f)$ and $J(f)$
  denote its Fatou and Julia sets. The set of critical points of $f$ is
  denoted $\Crit(f)=(f')^{-1}(0)$, and the postcritical set is
   \[ \P(f) := \cl{\bigcup_{j\geq 1} f^j\bigl( \Crit(f) \bigr)}. \]
 A point
  $z\in\Ch$ is \emph{exceptional} if 
  $\bigcup_n f^{-n}(z)$ is finite;
  there are at most two such points.

 \subsection*{Natural extension and regular leaf space}
  We denote by $\Nf$ the space of backward
  orbits of $f$. The function $f$ induces an invertible map
  $\fh:\Nf\to\Nf$ (whose inverse is the
  shift map); this map is called the \emph{natural extension} of $f$. 
  Whenever $A\subset\Ch$ 
  satisfies $f(A)\supset A$, its
  \emph{invariant lift} $\Ah\subset\Nf$ is the set of all backward
  orbits that remain in $A$. Abusing notation slightly, we will refer to
  the invariant lift of the Julia set $J(f)$ also simply as ``the Julia set''. 

 If $\z =(z_0\mapsfrom z_{-1} \mapsfrom
  z_{-2} \dots)$ 
   is a backward orbit and $U_0$ is a 
   connected open
  neighborhood of $z_0$, then the 
  \emph{pullback} of $U_0$ along $\z$
  is the sequence 
   $U_0\gets U_{-1}\gets \dots$, where $U_{-k}$ is
  the component of $f^{-k}(U_0)$ 
  that contains $z_{-k}$. This pullback
  is called \emph{univalent} if 
  $f:U_{-(k+1)}\to U_{-k}$ is univalent 
  for all $k$, and 
  \emph{regular} if this is true for all 
  sufficiently large $k$. If 
  the pullback is univalent, then we also 
   say that $U_0$ 
  is \emph{univalent} along $\z$. We 
  will 
  denote the set of all backward orbits 
  of
  $z_0$ along which $U_0$ is univalent 
  by $\univ(z_0,U_0)$.

  The point $\z$ is called \emph{unbranched} if it does not pass 
  through any critical points; 
  it is called \emph{regular} if there is 
  some open connected neighborhood of $z_0$ whose pullback along $\z$ 
  is regular. Note that, if $\z$ is unbranched and regular, then there
  exists a neighborhood of $z_0$ that is univalent along $\z$.

 The set of all regular backward orbits is called the \emph{regular 
  leaf space} and denoted 
  $\Rf$. The topology of $\Rf$ as a subset of the infinite product 
  $\Ch^{\N}$ is called the \emph{natural topology} of $\Rf$. 
  If $\z\in\Rf$, then the path-connected component $L(\z)$ of 
  $\Rf$ containing $\z$ is called the \emph{leaf} of $\z$. Every leaf
  can be turned into a Riemann surface by using the 
  projections 
  $\pi_{-k}:U\to\Ch, \zetah\mapsto \zeta_{-k}$ as charts (for sufficiently
  large $k$).

 We shall call a backward 
  orbit $\z$ \emph{parabolic} or \emph{hyperbolic} depending 
  on whether $L(\z)$ is a parabolic or hyperbolic Riemann 
  surface. The set of all parabolic backward orbits is called the 
  \emph{affine leaf space} and denoted by $\Afn$. 
  We will use the following facts about regular points and their leaves;
  compare \cite{mishayair}. 

 \begin{itemize}
  \item If $z_{-n}\notin \P(f)$ for large $n$, then $\z$ is regular.
  \item Invariant lifts of 
   Cremer, attracting
   and parabolic cycles and of
   boundaries of rotation domains are 
   never regular.
  \item Every leaf that is not the 
   invariant lift of a Herman ring is
    a hyperbolic or parabolic plane.
  \item The periodic leaves associated to repelling 
   periodic points are always
    parabolic; in fact, they are the
    Riemann surface of the 
    classical K{\oe}nigs linearization 
    coordinate. 
  \item Similarly, for every repelling petal based at a parabolic
    periodic point, there is an associated parabolic leaf, uniformized by
    the Fatou coordinate. Every backward orbit converging to a parabolic
    orbit belongs to such a leaf.
  \item Any parabolic leaf intersects the Julia set by Picard's theorem. 
%  \item If $\P(f)$ is finite --- in particular, if $f$ is a
%    Chebychev or Latt\`es map --- then all leaves of $f$ are
%    parabolic. 
 \end{itemize}

 \subsection*{The Lamination $\Af$}   
  We will now describe how the affine orbifold lamination $\Af$ is obtained
   from the affine part of the regular leaf space. We note that this
   lamination will not be used until the end of   
   Section \ref{sec:localcompactness}. Even there, the main fact that is 
   utilized is Proposition \ref{prop:laminartopology} below, 
   which can be understood without
   the exact details of the construction of $\Af$. 

  The group of linear transformations $z\mapsto az$, $a\neq 0$, 
  acts on the space
  $\U$ of nonconstant
  meromorphic functions $\psi:\C\to\Ch$ by 
  precomposition. 
  Let $\Ua$ denote the quotient of $\U$ by this action. 
  $f$ acts on $\Ua$ by
  postcomposition, and we can form the inverse limit space 
  $\widehat{\Ua}$ of sequences
  $\psih = (\psi_0\gets \psi_{-1} \gets 
            \psi_{-2}\gets\dots)$ with
  $\psi_i = f\circ \psi_{i-1}$.

 Now if $\z\in\Afn$, then $L(\z)$ is a parabolic plane, so there exists
  a conformal isomorphism  $\phi:\C\to L(\z)$
  with $\phi(0)=\z$. Thus
  $\psi_k := \pi_{-k} \circ \phi$ defines 
  an element of $\widehat{\Ua}$; note that
  $\psi_k$ depends only on $\z$ 
  since $\psi$ is unique up to 
  precomposition with a linear
  transformation. 
  
 The orbifold lamination $\Af$ is now 
  defined as
  the closure in $\widehat{\Ua}$ of all such 
  sequences. In a slight abuse of
  notation, we will denote the sequence $\psi_k$ associated to $\z$
  also by $\z$ and thus not differentiate between $\Afn$ and its copy
  inside $\Af$. As suggested by its name, $\Af$ is again a lamination;
  its leaves are the (parabolic) one-dimensional orbifolds
  \[ L^{\lam}(\psih) := 
          \{ \psih\circ T_a: a\in\C\},  \]
  where $T_a(z)=z+a$ and $\psih\circ T_a$ is the sequence with entries given
   by
    $\psi_{-j}\circ T_a$.

 Note that there are now two topologies defined on $\Afn\subset\Rf$: the
  original (natural) topology and that induced from $\Af$, called
  the ``laminar'' topology; the latter topological space will be
  denoted by $\Afl$. Rather than working directly
  with the above definition of
  $\Af$, we can use a criterion from \cite{mishayair} that describes
  the topology of $\Afl$ simply in terms of the natural extension. 
  If $V$ and $W$ are two simply connected
  domains, let us say that $V$
  is \emph{well inside $W$} if $\mod(W\setminus V)\geq 2$.

 \begin{prop}[Laminar topology {\cite[Proposition 7.5]{mishayair}}] 
    \label{prop:laminartopology}
  A sequence of points $\z^k\in \Afl$ converges to $\zetah\in\Afl$ in the
  laminar topology if and only if
  \begin{enumerate}
   \item $\z^k\to\zetah$ in the natural topology and
   \item for any $N>0$, if $V$ and $W$ are
    simply connected
    neighborhoods of $\zeta_{-N}$ such that 
    $\f^{-N}(\zetah)\in \univ(\zeta_{-N},W)$ 
    and $V$ is well  inside $W$,
    then 
      $\f^{-N}(\z^k)\in\univ(z_{-N},V)$ for large enough $k$. \qedd
     \label{item:pullbacks}
  \end{enumerate}
 \end{prop}
 \begin{remark}
  Condition (\ref{item:pullbacks}) is formally weaker 
    than that given
    in \cite{mishayair}; however, the proof remains the same.
 \end{remark}

 \subsection*{Auxiliary results}
 We will occasionally use the following classical fact, which is
 a weak version of
 the \emph{Shrinking Lemma} (see
  e.g.~\cite[Appendix 2]{mishayair}). 

 \begin{lem}[Univalent shrinking lemma] \label{lem:shrinking}
  Suppose that $U$ is a domain 
  univalent
  along some backward orbit $\z$
  that does not lie in the invariant
  lift of a rotation domain. Let
  $V_0\Subset U$ and denote by 
  $V_0\gets V_{-1}\gets \dots$ the
  pullback of $V_0$ along $\z$. Then
  $\diam V_{-j}\to 0$ (where $\diam$ denotes spherical diameter). \qedd
 \end{lem}

 Also, we will be concerned with the existence of unbranched backward
  orbits of a point $z\in\Ch$ under $f$. Let us say that
  $z$ is a \emph{branch exceptional point} if $z$ has at most finitely
  many unbranched backward orbits, and denote the set of
  such points by $E_B$.

 \begin{lem}[Branch exceptional points] \label{lem:branchexceptional}
  $E_B$ contains at most four points. If $z_0\notin E_B$, then
   for every $z\in J(f)$, there is some $w$ arbitrarily close to
   $z$ such that $f^n(w)=z_0$ and $(f^n)'(w)\neq 0$ for some $n$.
 \end{lem}
 \begin{proof}
  Every preimage of a point in $E_B$ either belongs to $E_B$ or is
   a critical point of $f$. Setting $d=\deg(f)$, it follows that
     \[ d\cdot(\# E_B) - \sum_{c} (\deg(c)-1)
        = \# f^{-1}(E_B) \leq \# E_B + \#(\Crit(f)\cap f^{-1}(E_B)), \]
   where the sum is taken over $c\in\Crit(f)\cap f^{-1}(E_B)$. 
   Since $f$ has exactly $2d-2$ critical points, counting
   with multiplicities, it follows that
     \[ (d-1) \# E_B \leq \sum_c \deg(c) \leq 4d-4, \]
    and hence $\# E_B\leq 4$. 

   If $z\notin E_B$, then $z$ has a
    non-periodic unbranched
    backward orbit $z\mapsfrom z_{-1} \mapsfrom z_{-2} \dots$.
    If $n$ is large enough, then $z_{-n}$ is not
    on a critical orbit, and any backward orbit of $z_{-n}$ is
    unbranched. Since iterated preimages of $z_{-n}$ are
    dense in the Julia set, the claim follows. 
 \end{proof}
 \begin{remark}
  For an alternative
    proof which applies also to transcendental meromorphic functions
   (using Nevanlinna's theorem on completely branched values),
   see \cite[Lemma 5.2]{linefields}.
 \end{remark}

 Branch exceptional points have a relation to the existence of
  isolated leaves in the affine orbifold lamination $\Af$. More
  precisely, if $p$ is a repelling periodic point in $E_B$, then 
  the periodic leaf $L(\hat{p})$ associated to the periodic
  backward orbit $\hat{p}$ of $p$ 
  is isolated in $\Af$. (This is an immediate consequence of
  Proposition \ref{prop:laminartopology}.)\footnote{%
   In \cite[Proposition 7.6]{mishayair}, it is stated (incorrectly)
   that such isolated
   leaves can only occur for Latt\`es and Chebyshev polynomials.
   Proposition \ref{prop:minimality} below provides
   a corrected version of this assertion.} 
  Proposition \ref{prop:minimality} shows 
   that these are the only examples
   of isolated leaves. In particular, every isolated leaf is periodic. 

  The famous \emph{Latt\`es} and \emph{Chebyshev} examples have branch
   exceptional repelling fixed points. In fact, for polynomials, we can
   give a description of maps with branch-exceptional periodic points.
   Indeed, suppose that $p$ is such a polynomial; then a calculation
   analogous to the proof of the previous lemma shows that $p$ has
   at most two branch-exceptional points apart from the superattracting
   fixed point at $\infty$. Furthermore,
   if $E_B(p)=\{\zeta_1,\zeta_2\}$ with $\zeta_1\neq\zeta_2$,
   then the set of critical points of $p$ coincides exactly with 
   $p^{-1}(E_B(p))\setminus E_B(p)$. It is well-known 
   (see e.g.\ \cite[Proposition 9.2]{douadyhubbardthurston})
   that this implies $p=T$ or $p=-T$, where $T$ is a Chebyshev polynomial. 
 
  On the other hand, if $E_B(p)$ consists of a single fixed point $z_0$, then
   every preimage of $z_0$ is a critical point, and hence $p$ is conjugate
   to a polynomial of the form
     \[ p(z) = z^{n}\cdot (z-a_1)^{k_1}\cdot\dots\cdot (z-a_m)^{k_m}, \]
   where $n\geq 1$, $a_j\in\C\setminus\{0\}$ and $k_j\geq 2$. 

  For rational maps, there are many more possible combinatorics for
   branch-exceptional periodic points. Even among Latt\`es maps, one can
   find functions with periodic points in $E_B$ which have
   periods $2$, $3$ and $4$, with a number of different combinatorial
   configurations. (Compare \cite{jacklattes}.) 

   Other examples are given e.g.\ by the family
    \[ f_{c}(z) := z\cdot \frac{z-1}{z^2-z-\frac{1}{\lambda}}, \quad
        \lambda\in\C\setminus\{0\}. \]
   Indeed,
    $\infty$ is a critical point of $f_c$, with
      \[ \infty \mapsto 1 \mapsto 0 \mapsto 0. \]
    Since $f_c$ is a quadratic rational map, it follows that 
    every non-periodic backward orbit of $0$ passes through the
    critical point $\infty$, and hence $0\in E_B$. Note that
    the fixed point $0$ has multiplier $\lambda$; in particular 
    $0$ may be attracting, parabolic or irrationally indifferent.

 \subsection*{Quadratic-like maps and renormalization}
  We
   quickly review the concepts regarding renormalization of
  quadratic polynomials relevant for Section \ref{sec:hypleaves}; 
  compare e.g.\ \cite{polylikemaps} for details.
  A \emph{quadratic-like map} is a proper map
  $\phi:U\to V$ of degree $2$, where $U$ and $V$ are Jordan domains
  with $U\Subset V$. The \emph{filled Julia set} of $\phi$ is
  \[ K(\phi) := \{z\in U: \phi^n(z)\in U \ \text{for all $n$}\}. \]

 By Douady and Hubbard's \emph{Straightening Theorem}, for every 
  quadratic-like map $\phi$ there exists a quadratic polynomial
  $f$ (the \emph{straightening} of $\phi$)
  that, restricted to a neighborhood of its filled
  Julia set, is (quasiconformally) conjugate to $\phi$. 

 A quadratic polynomial $g$ is called \emph{renormalizable} if there
  exists $n\geq 2$ and $U\subset\C$ such that
  $\phi := g^n|_U$ is a quadratic-like map with $K(\phi)$ connected.
  If $f$ is the straightening of $\phi$, then $g$ is also called a
  \emph{tuning} of $f$.

 It is well-known that every quadratic polynomial $f$ has
  (infinitely many) tunings; compare e.g.\ 
  \cite[Section 3]{jackpuzzle} for
  a precise statement.

\section{Existence of Hyperbolic Leaves} \label{sec:hypleaves}
  \newcommand{\B}{\mathcal{B}}

 Our proof of Theorem \ref{thm:quadratichyperbolic} begins
  with a result that establishes the existence of rational
  functions with many hyperbolic leaves that do \emph{not} intersect
  the Julia set. Recall that a subset of a Baire space is called
  \emph{generic} if it is the countable intersection of open dense sets. 

 \begin{thm}[Maps with large postcritical set] \label{thm:hyperbolicleaves}
  Let $f$ be a rational map and 
  suppose that $J(f)\subset \P(f)$. Let
  $z_0\in F(f)$ be a non-exceptional 
  point (i.e., a point that has infinitely many backward orbits).

  Then, for a generic
  backward orbit $\z$ of $z_0$, the 
  leaf $L(\z)$ does not intersect
  the Julia set. (In particular, $\z$ is 
  hyperbolic.)
 \end{thm}
 \begin{proof} Let $D$ be the Fatou component containing $z_0$. We first
  prove
  the theorem under the assumption that $D$ is not a rotation domain
  and that $z_0\notin\P(f)$. Below, we
  indicate how this implies the general case.

 \smallskip

  Denote the space of all backward orbits of $z_0$ by $\Z\subset\Nf$;
  note that by assumption these are all unbranched and regular. 
  Note also that
  $\P(f)\cap D$ is countable and has at most one accumulation point in
  $D$,
  which is then necessarily an attracting periodic point.

 Suppose that
   $\gamma:[0,1]\to D\setminus \P(f)$ is a curve with $\gamma(0)=z_0$.
  If $\z\in\Z$, let us denote by $\gamma(\z)$
  the endpoint of the corresponding
  lift of $\gamma$ to 
  $\Nf$. Note that the holonomy
    \[ \z \mapsto \gamma(\z) \]
  is a homeomorphism between the space of backward orbits of $z_0$ and 
  the space of backward orbits of $\gamma(1)$.

  Let $U\subset\Ch$ be any connected open set
   with $U\cap \partial D\neq\emptyset$ and let $\gamma$ 
   be a curve as above that satisfies $\gamma(1)\in U$. 
   Consider the set 
   $A_{U,\gamma,n}$ of all backward orbits $\z\in \Z$ for which the pullback
   of $U$ along $\gamma(\z)$ passes through a critical point at
   least $n$ times. Clearly the set $A_{U,\gamma,n}$ is open; by
   the density of $\P(f)$ in the Julia set, 
   it is also dense.

  It follows that the set 
    $A_{U,\gamma} := \bigcap_{n} A_{U,\gamma,n}$,
  which consists of all $\z\in \Z$ for which the pullback of $U$ along
    $\gamma(\z)$ is not regular, is generic.

  Now note that the pullback $\gamma(\z)$ depends only on the homotopy
   class of $\gamma$ in $D\setminus \P(f)$. Since the fundamental
   group of $D\setminus \P(f)$ is countable, and since $D\cap U$ has
   only countably many components, the set
    $A_U := \bigcap_{\gamma} A_{U,\gamma}$
   is also generic. 

  Finally, let $U_j$ be a countable collection of open sets
  such that $\{U_j\cap \partial D\}$ is a base for the topology of
  $\partial D$. Then
   \[ \A := \bigcap_{j} A_{U_j} \]
  is generic. We claim that $\pi_0(L(\z))\subset D$ for all
  $\z\in \A$ (that is, $L(\z)$ does not intersect the Julia set).

  Indeed,
  otherwise there would be a curve $\gammah:[0,1]\to L(\z)$ such that 
  $\gammah(0)=\z$; $\pi_0(\gammah(t))\in D$ for $t\neq 1$ and
  $\pi_0(\gammah(1))\in \partial D$. 
  Let $\gamma := \pi_0\circ\gammah$; the curve
  $\gammah$ can be easily chosen so that $\gamma\bigl([0,1)\bigr)\cap
  \P(f)=\emptyset$. Since $\gammah(1)$ is regular, there exists some
  small neighborhood $U$ of $\gamma(1)$ whose pullback along
  $\gamma(\z)$ is regular. This contradicts the construction of $\A$.

 \smallskip

  To conclude, consider the case where $z_0$ lies in
  a rotation domain or in the postcritical set. Suppose that 
  $z_{-n}\in f^{-n}(z_0)$ is a preimage that does
  not lie in a rotation domain or in the postcritical set. Then, for a
  generic point in the set
  $\Z(z_{-n})$ of all backward orbits of
  $z_{-n}$, the corresponding leaf
  does not intersect the Julia set. Therefore the same
  is true of a generic point in $\fh^n(\Z(z_{-n}))$. There is at most
  one backward orbit $\z^0$ of $z_0$ that belongs to the invariant lift
  of a rotation domain or of the 
  postsingular set (recall that $z_0\in F(f)$). Furthermore,
  $\Z\setminus\{\z^0\}$ can be written as the disjoint union
  of (countably many)
  sets of the form $\fh^n(\Z(z_{-n}))$. 
  The claim follows. \end{proof}

 The hyperbolic leaves 
  produced by the preceding theorem
  do not intersect the Julia set.
  It seems plausible that under the
  same hypotheses,
  there also exist some hyperbolic leaves
  that do intersect the Julia set.

 Instead, we will use Theorem \ref{thm:hyperbolicleaves} and the
  notion of tuning to prove Theorem \ref{thm:quadratichyperbolic}.
  (This idea is due to
  Rivera-Letelier.)

 \begin{prop}[Hyperbolic leaves over 
         the Julia set] 
                                    \label{prop:hyperbolicleavesjulia}
  Let $f$ be a quadratic polynomial whose regular leaf space contains
  a hyperbolic leaf over the basin of infinity. Then any tuning $g$
  of $f$ has
  a hyperbolic leaf that intersects the 
  Julia set. 
 \end{prop}

 We will use the following fact.  

 \begin{lem}[Almost every radial line lifts] \label{lem:gross}
  Let $f$ be a polynomial and let $z_0\in\C$. If $\z$ is an unbranched
  backward
  orbit of $z_0$ that belongs to a parabolic leaf, then for
  almost every $\theta\in \mathbb{R}/\mathbb{Z}$, the line
    \[ R_{\theta} := \{z_0+re^{2\pi i \theta}: r \geq 0 \} \]
  lifts to a curve in $L(\z)$ starting at $\z$.
 \end{lem}
 \begin{proof} Let $\phi:\C\to L(\z)$ be a conformal isomorphism with
  $\phi(0)=\z$. Consider the entire function $\psi := \pi_0\circ
  \phi$. By the Gross star theorem \cite[Page 292]{nevanlinna}, 
  the branch $\alpha$
  of $\psi^{-1}$ that
  carries $z_0$ to $0$ can be continued along almost every radial ray
  $R_{\theta}$. The curve $\phi(\alpha(R_{\theta}))\subset L(\z)$ 
  is then the required
  lift of $R_{\theta}$. \end{proof}

 \begin{proof}[Proof of Proposition \ref{prop:hyperbolicleavesjulia}]
  By assumption, there exists a domain $U$ such that 
   $\phi := g^n:U\to g(U)$ is quadratic-like and conjugate to $f$
   for some $n\geq 2$. Let us denote
   the straightening conjugacy by $h:U\to\C$.

  Note that every leaf that intersects the basin of
   infinity must also intersect the Julia set, since every backward orbit
   that does not belong to the invariant lift of 
    \[ \bigcup_{j=0}^{n-1} g^j(K(\phi)) \quad (\supset \P(f)) \]
   is regular.

  Since $K(\phi)$ is connected,
   we can find some point $z_0\in U\setminus K(\phi)$ such that the set
  \begin{align*}
    T:= \{\theta: \exists r_0>0:\ &z_0+r_0e^{2\pi i\theta}\in K(\phi) 
                                 \text{ and } \\
                                 &z_0+re^{2\pi i \theta}\in U\setminus 
                                 K(\phi)
                                  \text{ for $0\leq r < r_0$} \} 
  \end{align*}
  contains a nondegenerate interval.

  By assumption, $h(z_0)$ has a backward
  orbit whose leaf does not extend to $J(f)$ at all. Let $\z$ be the
  corresponding backward orbit under $g$. Then, for $\theta\in T$, the
  radial ray at angle $\theta$ starting in $z_0$ does not lift to
  $L(\z)$. By Lemma \ref{lem:gross}, this implies that $L(\z)$ is
  hyperbolic. \end{proof}

 \begin{remark}
  As noted in the introduction, the idea for an 
   example of a hyperbolic leaf
   that does not arise from a rotation domain was first suggested
   by Hubbard. 
   He proposed constructing a cubic polynomial with a superattracting
   fixed point at $0$ and a recurrent critical point in the boundary 
   of the basin of attraction of $0$, carefully chosen to make sure that
   some leaf does not extend beyond this basin of attraction. 
 \end{remark}

 \section{Failure of Local Compactness} \label{sec:localcompactness}

  To prove Theorem
   \ref{thm:localcompactness},
  let us begin with the following 
  statement, which
  roughly asserts that, given the presence 
  of a hyperbolic leaf
  intersecting the Julia set, we can find 
  hyperbolic (and parabolic)
  leaves close to any backward orbit.

 \begin{prop}[Hyperbolic leaves near unbranched orbits] 
    \label{prop:hyperbolicdensity}
  Let $f$ be a rational function and let $L$ be a leaf 
  of $\Rf$ that intersects the Julia set. 
  Let $\z\in\Rf$ be any unbranched
  backward orbit of $f$ that does not lie in the 
  invariant lift of
  a rotation
  domain of $f$, and let $V$ be an open simply connected
  neighborhood of $z_0$ that is univalent along $\z$. Assume furthermore
  that $\z$ does not belong to the 
  (isolated) periodic leaf of a branch-exceptional
  repelling
  periodic point.

  Then, for every domain $V_0\Subset V$ with $z_0\in V_0$, and for 
   any neighborhood $N$ of $\z$ in the natural topology, there
   is $m\in\N$ such that 
   $N\cap\univ(z_0,V_0)\cap \f^{m}(L)\neq\emptyset$.
 \end{prop}
 \begin{remark} The condition that $L$ 
  intersects the Julia set is clearly necessary, since $V$ itself
  may intersect the Julia set.
 \end{remark}
\begin{proof}
  Let $\hat{W}$ be an open subset of $L$ such that
   $W := \pi_0(\hat{W})$ is a simply connected domain 
   intersecting $J(f)$ and such that
   $\hat{W}$ is a univalent pullback of $W$ (i.e., every backward
   orbit in $\hat{W}$ is unbranched). 

  Under the hypotheses of the proposition, let $V_{-n}$ be the
   component of $f^{-n}(V_0)$ containing $z_{-n}$. We will show that
   there are infinitely many $n$ for which there is a univalent
   branch of $f^{-j}$ (for some $j\in\N$) that takes $V_{-n}$
   to a subset
   of $W$. This will complete the proof, as we can then continue
   pulling this subset back along the (univalent) pullback $\hat{W}$
   of $W$. 

 Let $A$ be the  limit set of $z_{-n}$ as $n\to\infty$; note that
   $A\subset J(f)$. We distinguish two cases. 

 \emph{First case:} $A$ is contained in the branch exceptional set
   $E_B$. Since $E_B$ is finite by Lemma \ref{lem:branchexceptional}, 
   $A$ then
   consists of a single periodic orbit; let $p\in A$ be a point of
   this orbit. By assumption, $p$ is not repelling. Clearly $p$ cannot
   be attracting, and by a result of Perez-Marco \cite{perezmarcodulacfatou}, 
   $p$ is not an irrationally indifferent orbit. Hence $p$ must
   be parabolic, and the point $\z$ belongs to the periodic leaf
   associated to some repelling petal based at this orbit. For simplicity,
   let us assume that $p$ is fixed (the periodic case is analogous). 
   Note that $p\notin V$ --- indeed, the only unbranched backward orbit 
   of $p$ is its invariant lift $\hat{p}$, and an invariant lift of
   a parabolic periodic orbit is never regular.

  Since the backward orbit of $p$ is dense in $J(f)$, we can find
   some $w\in W$ such that $f^j(w)=p$ for some $j$. If $\eps>0$
   is sufficiently small, then $w$ has a neighborhood
   $U\subset W$ such that $f^j:U\to D_{\eps}(p)$ has no critical
   points except $w$. For sufficiently large $n$,
   $V_{-n}$ is contained in a repelling petal $P$
   that is itself a simply connected subset of 
   $D_{\eps}(p)$. Hence there is a branch of
   $f^{-j}$ defined on $P$, and hence on $V_{-n}$, that takes values in
   $U\subset W$. This completes the proof in this case.  

  (We recall that the family of rational functions $f_C$
   given in Section \ref{sec:preliminaries} contains
   maps with branch exceptional 
   parabolic points, so this case may indeed occur.)

  \emph{Second case:} There is some $a\in A\setminus E_B$. Then by
   Lemma \ref{lem:branchexceptional}, there is some $j\in\N$ and $w\in W$ 
   such that
   $f^j(w)=a$ and such that $w$ is not a critical point of $f^j$.
   Let $\eps>0$ be sufficiently small such that
   the component $U$ of $f^{-j}(D_{\eps}(a))$ is contained in $W$
   and $f:U\to D_{\eps}(a)$ is a conformal isomorphism.

  If $z_{-n_k}$ is a subsequence of $\z$ with $z_{-n_k}\to a$, then
   by Lemma \ref{lem:shrinking}, $\diam V_{-n_k}\to\infty$ as $n\to\infty$,
   and hence $V_{-{n_k}}\subset D_{\eps}(a)$ for sufficiently large $k$.
   Again, we see that there is a branch of $f^{-j}$ defined on
   $V_{-n_k}$ that takes values in $W$, and are done.
 \end{proof}

 The reason that the presence of hyperbolic backward orbits
  leads to failure of local
  compactness is that parabolic leaves accumulating at such an orbit
  cannot converge, even in the weaker topology of $\Af$:

 \begin{lem}[No convergence to hyperbolic leaves]
   \label{lem:hyperbolicaccumulation}
  Let $\z^n$ be a sequence of points in $\Afl$
   that converges, in the natural topology of $\Rf$,
   to a point $\zetah$ for which
   $L(\zetah)$ is hyperbolic.
  Then $\z^n$ does not converge in $\Af$.
 \end{lem}
 \begin{proof} Suppose that $\psih$ is a limit point of $\z^n$ in $\Af$. Then
 $\psi_{-n}(0)=\zeta_{-n}$ for all $n$. However, this means that the
 projection 
   \[ p: L^{\lam}(\psih)\to L(\zetah);\ 
         \bigl(\phi_0\gets \phi_{-1}\gets \dots\bigr) \mapsto  
         \bigl(\phi_0(0)\gets \phi_{-1}(0)\gets\dots\bigr) \]
 is a nonconstant 
 holomorphic map from the affine orbifold $L^{\lam}(\psih)$ to
 the hyperbolic surface $L(\zetah)$ (cf.\ \cite[\S
 6.1]{mishayair}). This is impossible. \end{proof}

 After these preliminaries, we are ready to
  prove Theorem \ref{thm:localcompactness}. 

 \begin{lem}[Non-pre-compact boxes] \label{lem:noncompactbox}
  Suppose that $f$ has a hyperbolic leaf $L_h$ that intersects
   the Julia set. 
   Let $\zetah\in \Afn$ be an unbranched backward orbit that does not
   belong to an isolated leaf of $\Af$.

  Let
   $V$ and $W$ be simply connected neighborhoods of $\zeta_0$
   such that
   $V$ is well inside $W$ and $W$ is univalent along $\zetah$. 
   Then $\univ(\zeta_0,V) \cap \Afn$ is not
   pre-compact in $\Af$.
 \end{lem}
 \begin{proof}
  Choose some $V_0$ with 
   $V \Subset V_0\Subset W$.
   By Proposition 
   \ref{prop:hyperbolicdensity}, there 
   exists a
   backward orbit 
   $\z \in \univ(\zeta_{0},V_0)$ such that the leaf $L(\z)$ is hyperbolic. 
   Now we can apply
   Proposition 
   \ref{prop:hyperbolicdensity} to
   $\z$, this time 
   with $L$ being a parabolic leaf.  Hence 
   there is a
   sequence 
   $\z^k\in\univ(\zeta_{0},V_0)
      \cap\Afn$ 
   that converges to
   $\z$ in the natural topology. By 
   Lemma
   \ref{lem:hyperbolicaccumulation}, this 
   sequence has no convergent
   subsequence in $\Af$. 
 \end{proof}

 \begin{cor}[Failure of local compactness] \label{cor:nonlocalcompactness}
  Let $f$ be a rational function and suppose that $\Rf$ contains 
  some hyperbolic leaf
  $L$ that intersects the Julia set. 
  Then, for all $\zetah\in\Af$ that are not on isolated leaves,  
  $\Af$ is not locally compact at
  $\zetah$.
 \end{cor}
 \begin{remark}
  This completes the proof of Theorem
   \ref{thm:localcompactness}, and thus of Theorem \ref{thm:main}.
 \end{remark}
 \begin{proof} 
  By definition, $\Afl$ is dense in
   $\Af$. So it is sufficient to restrict
   to the case of
   $\zetah\in\Afl$. Also, unbranched 
   backward orbits are dense in
   $\Afl$, so we can assume that
   $\zetah$ is unbranched.

  By Proposition
  \ref{prop:laminartopology}, the sets 
  $\f^n\bigl(\univ(\zeta_{-n},V)
         \bigr)$, where
  $n\geq 0$ and $V,W$ are Jordan
  neighborhoods of $\zeta_{-n}$ such that 
  $W$ is univalent along $\f^n(\zeta_{-n})$
  and $V$ is well inside $W$, form
  a neighborhood base of
  $\zetah$ in the laminar topology of $\Afl$. 
  By Lemma
  \ref{lem:noncompactbox}, none of these sets
  is pre-compact in $\Af$. So $\Af$ is not
  locally compact at $\zetah$. 
  \end{proof}

 Finally, we remark that Proposition \ref{prop:hyperbolicdensity} also
  proves minimality of the lamination $\Af$ (after removing finitely many
  isolated leaves). 

 \begin{prop}[Minimality] \label{prop:minimality}
  Let $L$ be a leaf of $\Af$, and let $\z\in\Afl$ be a point that
   does not belong to the isolated leaf associated to a branch-exceptional
   repelling periodic point. Then $L$ accumulates on $\z$ in the topology
   of $\Af$. 

  In particular, let $\Af'$ be obtained from
   $\Af$ by removing all (finitely many) leaves associated to 
   branch-exceptional repelling orbits. 
   Then $\Af'$ is minimal; i.e., every leaf of $\Af'$ is
   dense in $\Af'$. 
 \end{prop}
 \begin{remark} 
  By Lemma \ref{lem:branchexceptional}, there are at most four 
   branch-exceptional periodic points, so the lamination $\Af$ contains
   at most four isolated leaves. This bound is achieved for some 
   Latt\`es maps. 
 \end{remark}
 \begin{proof}
   (Compare \cite[Proposition 7.6]{mishayair}.) If $L$ is an invariant
    leaf in $\Afl$, then the claim follows immediately from Propositions
    \ref{prop:hyperbolicdensity} and \ref{prop:laminartopology}. 

  Now suppose that $L=L(\psih)$ is an arbitrary leaf of $\Af$; by passing to
   an iterate, we may assume that $f$ has at least five repelling
   fixed points $\alpha_1,\dots,\alpha_5$ that are not branch-exceptional. 
   We show that $L$ accumulates at the invariant leaf
   of at least one of these fixed points, which completes the proof.
   
  To do so, let $D_k$ be pairwise
    disjoint linearizing Jordan neighborhoods of the 
    fixed points $\alpha_k$. By the Ahlfors Five Islands theorem 
    (see \cite{walterahlfors}), for every $j$ there is some $k_j$ such that
    $\psi_{-j}:\C\to\Ch$ has an \emph{island} over $D_{k_j}$. That is,
    there is a domain $V_j$ such that $\psi_{-j}:V_j\to D_{k_j}$ is 
    a conformal isomorphism. 

   There is some $k$ such that $k_j=k$ for infinitely many $j$. 
    Analogously to the proof of
    \cite[Proposition 7.5]{mishayair}, it follows that
    that $L$ accumulates on the invariant leaf $L(\hat{\alpha_k})$, 
    as desired.
 \end{proof}

\nocite{jackdynamicsthird}
\bibliographystyle{hamsplain}
% unsrt to have [1]unsorted, plain to have [1] sorted, alpha to have
% something horrible, abbrv to have the same but shorter
\bibliography{/Latex/Biblio/biblio}

\end{document}